\documentclass[11pt]{article}
\usepackage{latexsym,amsfonts,euscript,amsmath}
\usepackage{amssymb}
\usepackage{amsthm}
\usepackage{stmaryrd}
\usepackage{mathrsfs}
\usepackage{mathabx}
\usepackage{leftidx}
\usepackage{indentfirst}
\usepackage[section]{placeins}

\usepackage[colorlinks,
linkcolor=blue,
anchorcolor=blue,
citecolor=blue]{hyperref}

\makeatletter 
\@addtoreset{equation}{section}
\makeatother  

\def\la{\langle}\def\ra{\rangle}

\newcommand{\GL}{{\operatorname{GL}}}

\newcommand{\PSL}{{\operatorname{PSL}}}
\newcommand{\PSU}{{\operatorname{PSU}}}

\newcommand{\soc}[1]{\mathrm{Soc}(#1)}

\newcommand{\cp}{\mathfrak{C}(p)}
\newcommand{\ncp}{\mathfrak{C}^{\sharp}(p)}
\newcommand{\cpm}[1]{\mathfrak{C}(p^{#1})}
\newcommand{\ncpm}[1]{\mathfrak{C}^{\sharp}(p^{#1})}

\def\irr#1{{\rm Irr}(#1)}
\def\cent#1#2{{\bf C}_{#1}(#2)}
\def\syl#1#2{{\rm Syl}_#1(#2)}
\def\norm#1#2{{\bf N}_{#1}(#2)}
\def\oh#1#2{{\bf O}_{#1}(#2)}
\def\ohphi#1{{\bf O}_{p',\Phi}(#1)}

\def\fitt#1{{\bf F}(#1)}
\def\gfitt#1{{\bf F}^*(#1)}
\def\E#1{{\bf E}(#1)}
\def\z#1{{\bf Z}(#1)}

\def\o#1{\overline{#1}}

\newtheorem{lem}{ \bf Lemma}[section]

\newtheorem{prop}[lem]{\bf Proposition}
\newtheorem{thm}[lem]{\bf Theorem}
\newtheorem*{thm*}{\bf Theorem}
\newtheorem*{thmA}{\bf Theorem A}
\newtheorem*{thmB}{\bf Theorem B}
\newtheorem*{corC}{\bf Corollary C}

\newtheorem{cor}[lem]{\bf Corollary}

\title{On finite groups with certain complemented $p$-subgroups
\thanks{{\bf Acknowledgement:} The author gratefully acknowledge the support of China Scholarship Council (CSC).}
}

\author{Yu Zeng\\
{\footnotesize\small  Dept. Mathematics, Changshu Institute of
Technology, Changshu, Jiangsu, 215500, China}\\
{\footnotesize\small E-mail: yuzeng2004@163.com}}

\begin{document}
\maketitle
\date{}

\vskip 1cm

\begin{center}\textbf{Abstract}\end{center}


Given a prime power $p^d$ with $p$ a prime and $d$ a positive integer,
we classify the finite
groups $G$ with $p^{2d}$ dividing $|G|$ in which all subgroups of order $p^d$
are complemented and 
the finite groups $G$ having a normal elementary abelian Sylow $p$-subgroup $P$ such that $p^d<|P|$ in which all subgroups of order $p^d$ are complemented.

\vskip 5cm

\bigskip


\pagebreak

\section{Introduction}

A \emph{complement} of a subgroup $H$ in a group $G$ is a subgroup $K$ of $G$ such that
$G=HK$ and $H\cap K=1$ (here $H$ is called a \emph{complemented subgroup} of $G$).
A classical result by P. Hall states that if every Sylow subgroup of a finite group $G$ has 
a complement in $G$, then $G$ is solvable.
In \cite{hall1937}, P. Hall also characterized finite groups $G$ in which every subgroup of $G$ has a complement.
Motivated by Hall's work, a number of authors studied the structure of finite groups $G$ in which certain class of subgroups are complemented.
For instance, Y.M. Gorchakov showed that Hall’s requirement of the
complementability of all subgroups can be reduced to the complementability of all minimal subgroups (see \cite{gorchakov1960}). 
In this paper, we focus our attention on the complementability of certain $p$-subgroups.

Let $p$ always denote a prime, $d$ always denote a positive integer and $G$ always denote a finite group.
We say that a finite group $G$ is a \emph{$\cpm{d}$-group} if each subgroup of $G$ of order $p^d$ is complemented in $G$;
if $G$ is a $\cpm{d}$-group such that $p^d\leq |G|_p$, then we say $G$ is a \emph{nontrivial $\cpm{d}$-group}.
There has been a lot of interest in the problem of characterizing nontrivial $\cpm{d}$-groups.

V.S. Monakhov and V.N. Kniahina 
investigated the $pd$-composition factors (the composition factors
of order divisible by $p$) of nontrivial $\cp$-groups (see \cite{monakhov2015}).
In \cite{qian2015}, G. Qian and F. Tang described $p$-solvable nontrivial $\cpm{d}$-groups $G$ for a given prime power $p^d\leq \sqrt{|G|_p}$; 
and also characterized nontrivial $\cpm{d}$-groups $G$ which has a normal elementary abelian Sylow $p$-subgroup for a given prime power $p^d< |G|_p$. 
Having classified nontrivial $\cp$-groups and $\cpm{2}$-groups in \cite{zeng2019,zeng2020},
we in this paper 
classify nontrivial $\cpm{d}$-groups $G$ for a fixed prime power $p^d\leq \sqrt{|G|_p}$ and obtain the classification of nontrivial $\cpm{d}$-groups $G$ which has a normal elementary abelian Sylow
$p$-subgroup for a fixed prime power $p^d< |G|_p$.

\begin{thmA}\label{Thm1}
	Let $p^d$ be a power of a prime $p$ such that $p^d>1$ and $G$ a finite group 
	such that $|G|_p\geq p^{2d}$.
	Assume $\oh{p'}{G}=1$.
	Then $G$ is a nontrivial $\cpm{d}$-group if and only if one of following is true.

	{\rm (1)} $G$ is a nontrivial $\cp$-group. 

	{\rm (2)} $G=H\ltimes P$ where $H$ is a cyclic Hall $p'$-subgroup and $P\in\syl{p}{G}$ is a faithful homogeneous $\mathbb{F}_p[H]$-module with all its 
	irreducible submodules having dimension $e$$(>1)$ so that $e\mid (d,\log_p|P|)$.
\end{thmA}
Observe that $G$ is a $\cpm{d}$-group if and only if
$G/\oh{p'}{G}$ is a $\cpm{d}$-group (see part (3) of Lemma \ref{element31}). 
Hence, fixed a prime power $p^d$,
Theorem A and \cite[Theorem 1.1]{zeng2019} give a complete
classification of finite $\cpm{d}$-groups $G$ such that $p^d\leq \sqrt{|G|_p}$.

Next theorem is a key to prove Theorem A.

\begin{thmB}
	Assume that a finite $p'$-group $H$ acts faithfully on a finite elementary abelian $p$-group $V$.
	Let $G=H\ltimes V$ where $|V|=p^n$.
	Then $G$ is a nontrivial $\cpm{d}$-group for a fixed positive integer $d<n$ if and only if 
	one of the following statements holds.

	{\rm (1)} $G$ is supersolvable.

	{\rm (2)} $H$ is cyclic and
	$V$ is a homogeneous $\mathbb{F}_p[H]$-module with all its irreducible $\mathbb{F}_p[H]$-submodules having dimension $e$$(>1)$ such that $e\mid (d,n)$.
\end{thmB}


If $G$ is a finite group with $\oh{p'}{G}=1$ in which $P\in\syl{p}{G}$ is normal,
then by \cite[Lemma 2.4]{qian2020} $\Phi(P)=P\cap \Phi(G)=\Phi(G)$.
So by part (3) of Lemma \ref{element31} the next corollary, which is an improvement of \cite[Proposition E]{qian2015}, follows immediately.

\begin{corC}
	Assume that a finite $p'$-group $H$ acts faithfully on a finite $p$-group $P$.
	Let $G=H\ltimes P$ where $|P|=p^n$ and $\o G=G/\Phi(G)$ where $|\o P|=p^{n-s}$.
	If $G$ is a nontrivial $\cpm{d}$-group for a fixed positive integer $d<n$,
	then one of the following statements holds.

	{\rm (1)} $\o G$ is supersolvable.

	{\rm (2)} $H$ is cyclic and
	$\o P$ is a homogeneous $\mathbb{F}_p[H]$-module with all its irreducible $\mathbb{F}_p[H]$-submodules having dimension $e$$(>1)$ such that $e\mid (d-s,n-s)$.
\end{corC}

The paper is organized as follows: in Section 2 we collect useful results 
and also give a detailed description of simple nontrivial $\cpm{d}$-groups;
at the end of Section 3 we prove Theorems A and B; in Section 4 we improve \cite[Theorem A$'$]{qian2020} by applying Theorem B.

Our notation is standard, for group theory it follows \cite{huppert} and for character theory of finite groups it follows \cite{isaacs1976}.
All groups considered in the paper are finite.

\section{Preliminaries}

We begin with a couple of lemmas for later use.

\begin{lem}[\cite{qian2004}]\label{qian2004}
	Let $G$ be an almost simple group with socle $S$. 
	If $p$ divides both $|S|$ and $|G:S|$,
	then $G$ has non-abelian Sylow $p$-subgroups.
\end{lem}

\begin{lem}\label{ea}
	Let $G$ be a group such that $\oh{p'}{G}=1$ and $P\in\syl{p}{G}$. Then the following hold.  

	{\rm (1)} If $P$ is elementary abelian, then $\Phi(G)=1$.

	{\rm (2)} If $P$ is abelian, then $P\leq\gfitt{G}$.  
\end{lem}
\begin{proof}
	(1) Let $G$ be a minimal counterexample.
	Let $N$ be a minimal normal subgroup of $G$ and $K/N=\oh{p'}{G/N}$. 
	By the minimality of $G$, $\Phi(G/K)=1$. 
	Therefore $\Phi(G)\leq N$ as $\oh{p'}{G}=1$, so $\Phi(G)$ is the unique minimal normal subgroup of $G$.
	We now show that $\E G>1$.
	Otherwise, $\gfitt{G}=\fitt{G}=P$.
	So by \cite[Lemma 2.4]{qian2020} $\Phi(G)=\Phi(G)\cap P=\Phi(P)=1$, a contradiction.
	Let $E$ be a component of $G$.
	If $E<G$, then by the minimality of $G$ we have $\Phi(E)=1$.
	Further, by the minimality of $G$ and the uniqueness of minimal normal subgroup of $G$, 
	$G$ is quasisimple with $\z G=\Phi(G)\cong C_p$.
    Let $\lambda\in\irr{\z G}$ be non-principal. 
	By \cite[Theorem 6.26]{isaacs1976}, there exists a linear character $\chi\in\irr{G}$ such that $\chi_{\z{G}}=\lambda$, so $\lambda=1_{\z G}$ as $\chi=1_G$,
    a contradiction.

	(2) Let $G$ be a counterexample.
	Note that $\fitt{G}$ is a $p$-group as $\oh{p'}{G}=1$.
    If $\E G=1$, then $P=\gfitt{G}$ as $P$ being abelian, a contradiction.
    So $\E{G}>1$.
	Observing that $P\nleq \gfitt{G}$ and $P$ is abelian,
	we take a $p$-element $x\in G-\gfitt{G}$ such that $x^p\in \gfitt{G}$.
	Write $\overline{G}=G/\Phi(\E{G})$.
	Then $\la\overline{x}\ra$ acts on $\overline{\E{G}}=S_1\times \cdots \times S_t$,
	where $S_i$ are nonabelian simple groups.
	Also since $\oh{p'}{\overline{G}}=1$, $p\mid |S_i|$ for $1\leq i\leq t$.
    Since there exists a Sylow $p$-subgroup of $S_i$ which is centralized by $\la \overline{x}\ra$, $\overline{x}$ fixes each $S_i$. 
	We now show that there exists $1\leq i_0\leq t$ such that $[\overline{x},S_{i_0}]\neq 1$.
	Otherwise, $[\la x\ra,\E{G}]\leq \Phi(\E{G})$.
	Notice that $\Phi(\E{G})=\z{\E{G}}$, and so $[\la x\ra,\E{G},\E{G}]=[\E{G},\la x\ra,\E{G}]=1$.
    By Three subgroups lemma, $[\E{G},\la x\ra]=[\E{G},\E{G},\la x\ra]=1$,
    observing that $[\la x\ra, \gfitt{G}]=1$ as $P$ being abelian,
    and it follows from $\cent{G}{\gfitt{G}}\leq \gfitt{G}$ that $x\in \gfitt{G}$, a contradiction.
	Therefore $\la\overline{x}\ra S_{i_0}/\cent{\la \overline{x}\ra}{S_{i_0}}$ is an almost simple group with socle $S_{i_0}\cent{\la \overline{x}\ra}{S_{i_0}}/\cent{\la \overline{x}\ra}{S_{i_0}}$ and an abelian Sylow $p$-subgroup, which contradicts Lemma \ref{qian2004} as $\la\overline{x}\ra S_{i_0}/\cent{\la \overline{x}\ra}{S_{i_0}}>S_{i_0}\cent{\la \overline{x}\ra}{S_{i_0}}/\cent{\la \overline{x}\ra}{S_{i_0}}$.
	Thus $P\leq\gfitt{G}$.
\end{proof}

We now collect some useful results about $\cpm{d}$-groups. 
In the following, we denote the class of $\cpm{d}$-groups by $\cpm{d}$ and the class of nontrivial $\cpm{d}$-groups by $\ncpm{d}$.

\begin{prop}\label{simpleCpd}
	Let $S$ be a nonabelian simple group. 
	Then $S\in\ncpm{d}$ if and only if 
	$|S|_p=p^d$ and one of the following holds.
   
	{\rm (i)} $S=A_p$ and $p^d=p\geq 7$.

	{\rm (ii)} $S=\PSL(2,11)$ and $p^d=11$.

	{\rm (iii)} $S=\mathrm{M}_{11}$ and $p^d=11$.

	{\rm (iv)} $S=\mathrm{M}_{23}$ and $p^d=23$.

	{\rm (v)}  $S=\PSL(n,q)$ where $p^d=p=\frac{q^n-1}{q-1}$ and $n$, $p$ are distinct primes.

	{\rm (vi)} $S=\PSL(n,q)$ where $p^d=\frac{q^n-1}{q-1}>p>2$ and $n$, $p$ are different primes.
	In particular, Sylow $p$-subgroup of $S$ is isomorphic to $C_{p^d}$.

	{\rm (vii)} $S=\PSL(2,q)$ where $p^d=2^d=q+1\geq 8$ and $q$ is a Mersenne prime. In particular, 
	Sylow $2$-subgroup of $S$ is isomorphic to $D_{2^d}$.	
\end{prop}
\begin{proof}
	By \cite[Lemma 2.8]{zeng2019}, it suffices to prove that $G\in\ncpm{d}$  for $d>1$ if and only if either (vi) or (vii) holds. 
	
	Assume either (vi) or (vii) holds. Then by \cite[Theorem 1]{guralnick1983} 
	Hall $p'$-subgroup of $S$, which has index $p^d$ in $S$, exists.
	That is $S\in\ncpm{d}$ for $d>1$. 
	
    Assume now $S\in\ncpm{d}$ for $d>1$.
	Then $S$ has a subgroup of index $p^d$, and hence, by \cite[Theorem 1]{guralnick1983}, 
	$S\cong A_{p^d}$; or $S\cong \PSU(4,2)$ and $p^d=3^3$; or $S\cong \PSL(n,q)$ with $p^d=\frac{q^n-1}{q-1}$ and $n$ a prime. 
	
	Suppose $S\cong A_{p^d}$ with $d>1$. Then $|S|_p=p^n\geq p^{2d}$,
	and hence $S$ has an elementary abelian Sylow $p$-subgroup by \cite[Proposition F]{qian2015}, 
	a contradiction. 
	Suppose $S\cong \PSU(4,2)$ and $p^d=3^3$. 
	Let $D$ be a subgroup of order $3^3$ in $S$. 
	Then $S=D\cdot H$, where $D\cap H=1$.
	Therefore $S=D\cdot H^x$ so that $D\cap H^x=1$ for all $x\in S$.
	By checking \cite{atlas}, we know that all subgroups of index $3^3$ in $S$ are conjugate in $S$.  
	So we conclude a contradiction as $|H|_3=3$.
	
	Thus $S\cong \PSL(n,q)$ with $p^d=\frac{q^n-1}{q-1}$ and $n$ a prime.
	Then it follows from \cite[Theorem 1]{guralnick1983} that $|S|_p=p^d$.
	Also, by \cite[II, Theorem 7.3]{huppert} $S$ has a cyclic subgroup of order $\frac{q^n-1}{c(q-1)}=\frac{p^d}{(q-1,n)}$ where $c=(q-1, n)$ (this cyclic subgroup is called Singer cycle).
	Suppose $n\neq p$. 
	Then $S$ has a cyclic subgroup of order $p^d$, and hence (vi) follows.
	Suppose $n=p$. 
	Note that $p^d=\frac{q^n-1}{q-1}$,
    and so $q^p\equiv 1(\mathrm{mod}~p)$ and also $q^{p-1}\equiv 1(\mathrm{mod}~p)$.
    Hence $q=q^{(p,p-1)}\equiv 1(\mathrm{mod}~p)$.
    If $p>2$, then it follows by \cite[Lemma 10.11]{isaacs1976} that $p$ is the exact power of $p$ dividing $\frac{q^p-1}{q-1}=p^d$, contradicting to $d>1$.
    So $n=p=2$. 
	Thus $S\cong \PSL(2,q)$ with $q=2^d-1$.
	By \cite[Theorem 5]{mihailescu}, $q=2^d-1$ is a Mersenne prime. 
	In this case, Sylow $2$-subgroups of $S$ are isomorphic to $D_{2^d}$ where $d\geq 3$.
\end{proof}


\begin{lem}\label{element31} 
	For a group $G\in\cpm{d}$, the following hold.

	{\rm (1)} If $H\leq G$, then $H\in\cpm{d}$.

	{\rm (2)} If $N\leq G$ is a direct product of $n$ simple groups for $n\leq d$, then $|N|_p\leq p^d$.
	In particular, if $N$ is minimal normal in $G$, then $|N|_p\leq p^d$.

	{\rm (3)}  If $N$ is a normal subgroup of $G$  with $|N|_p=p^e\leq p^d$, then $G/N\in\cpm{d-e}$.
	Furthermore, if $e=0$, then $G/N\in\cpm{d-e}$ implies that $G\in\cpm{d}$.


	{\rm (4)} If $|G|_p\geq p^d$ and $N$ is subnormal in $G$ with $|N|_p=p^e$,
	then $N\in\cpm{m}$ where $m=\min\{d,e\}$.

	{\rm (5)} $G\in\cpm{md}$ for each nonnegative integer $m$.
\end{lem}

\begin{proof}
	For (1), (3), (4) and (5), we refer to \cite[Lemma 3.1]{zeng2020}. We only prove (2) here.

   By Proposition \ref{simpleCpd}, we may assume  that $n>1$.
   Let $G$ be a counterexample with smallest possible sum $|G|+n+d$.
   Then $|N|_p>p^d$, $G=N$ and each direct factor of $N$ has order divisible by $p$. 
   Write $N=S\times T$ where $S$ is simple and $T$ is a direct product of $n-1$ simple groups,
   we claim that $|T|_p<p^d$.
   Otherwise, $|T|_p\geq p^d$.
   Let $P\leq S$ be such that $|P|=p$ and let $H=P\times T$.
   Then by (1) $H\in\cpm{d}$ and hence by (3) $T\cong H/P\in\cpm{d-1}$.
   Notice that $T\leq G$ is a direct product of $n-1$ simple groups,
   so by the minimality of $G$ we have $|T|_p\leq p^{d-1}$, a contradiction.
   Now $|T|_p=p^a$ for some positive integer $a<d$.
   By (3), $S\cong N/T=G/T\in\cpm{d-a}$, so $|S|_p\leq p^{d-a}$.
   As a consequence, $|N|_p=|S|_p|T|_p\leq p^d$, a contradiction.
   
   In particular, if $N$ is minimal normal in $G$, then by \cite[Lemma 3.1(2)]{zeng2020} $N$ is a direct product of at most $d$ copies of simple groups, and so $|N|_p\leq p^d$.
\end{proof}

\section{Proof of the main theorem}

We denote by $\ohphi{G}$ the subgroup of a group $G$ such that $\ohphi{G}/\oh{p'}{G}=\Phi(G/\oh{p'}{G})$.
Let $\o G=G/\ohphi{G}$. 
Then $\oh{p'}{\o G}=\Phi(\o G)$.

\begin{lem}\label{equalppart}
  Let $G\in\ncpm{d}$ be such that $\oh{p'}{G}=\Phi(G)=1$ and $|G|_p=|\gfitt{G}|_p>p^d$.
  Assume that $|N|_p$ is a constant for each minimal normal subgroup $N$ of $G$.
  Then $\E G=1$.    
\end{lem}
\begin{proof}
  As $\Phi(G)=1$, $\gfitt{G}=\soc{G}$. 
  Also since $\oh{p'}{G}=1$, by part (2) of Lemma \ref{element31} 
  the $p$-part of the order of every minimal normal subgroup $p^e$ satisfies $1\leq e\leq d$.
  Let $G$ be a counterexample with smallest possible sum $|G|+d$. 
  Then there exists a nonabelian minimal normal subgroup $E$ of $G$.
  Note that $|\mathrm{Soc}(G)|_p=|G|_p>p^d$ and also $|E|_p=p^e\leq p^d$,
  and so there exists another minimal normal subgroup $D$ of $G$.
	
	Let now $P\in\mathrm{Syl}_p(D)$. We show that $D=P$.
    Write $\gfitt{G}=D\times E\times C$ where $C$ is a direct product of some minimal normal subgroups of $G$, and let $L=\norm{G}{P}$ and $\widetilde{L}=L/\ohphi{L}$.
	Observe that $\widetilde{P}\times\widetilde{E}\times\widetilde{C}\leq \gfitt{\widetilde{L}}$ and $|\widetilde{L}|_p=|\widetilde{P}\times\widetilde{E}\times\widetilde{C}|_p$,
	and so $\widetilde{L}$ satisfies the hypothesis of this lemma.
	Thus, by the minimality of $G$, $G=L=\norm{G}{P}$, that is $D=P$ is an abelian minimal normal subgroup of $G$.
    
	We next show that $e=d$.
	Assume $e<d$. 
	Let $N/D=\oh{p'}{G/D}$, $M/N=\Phi(G/N)$ such that $|M/N|_p=p^s$ and $\overline{G}=G/M$.  
	Then by part (3) of Lemma \ref{element31} $\overline{G}\in\ncpm{d-e-s}$.
	Since $E\cap M=1$, $\overline{E}$ is minimal normal in $\overline{G}$.
    It is routine to check that $\overline{G}$ satisfies the hypothesis of this lemma.
	So, by the minimality of $G$, $E\cong \overline{E}\leq \E{\overline{G}}=1$, a contradiction.
	
	Therefore $e=d$.
	Recall that $|D|=p^d$, and hence $G=D\rtimes H$, where $H\leq G$, as $G\in\ncpm{d}$.  
    We then show that $\gfitt{G}=D\times E$.
    To see this, as $\oh{p'}{G}=1$, it suffices to show $|\gfitt{G}|_p=p^{2d}$.
    Assume $|\gfitt{G}|_p>p^{2d}$.
    Then $|H|_p>p^d$, and $H\in\ncpm{d}$ by part (1) of Lemma \ref{element31}.
    Let $\widehat{H}=H/\ohphi{H}$.
	Clearly, $\widehat{H}$ satisfies the hypothesis of this lemma,
	and hence $\E{\widehat{H}}=1$ by the minimality of $G$.
	However, $\widehat{H}$ contains a nonabelian minimal normal 
    subgroup which is isomorphic to $E$ as $H\cong G/D$,
	a contradiction.
    Therefore $\gfitt{G}=D\times E$ where $D$ and $E$ ($\leq \E{G}$) are distinct minimal normal subgroups of $G$ so that $|D|=p^d$ and $|E|_p=p^d$.
	Let $X\leq D$ be of order $p^a$($>1$) and $Y\leq E$ be of order $p^{d-a}$($>1$).
	Hence $G=XY\cdot M$ where $M$ is a complement of $XY$ in $G$ as $G\in\ncpm{d}$. 
    Then $|G:M|=p^d$ and $G=(D\times E)M$.
    Since $XY\leq \gfitt{G}$, $\gfitt{G}\cap M$ is a complement of $XY$ in $\gfitt{G}$.
	Next we conclude a contradiction by proving $E$ is abelian.
	To see that, it suffices to show $E\cap M=1$.
	We know that $E\cap M\unlhd DM$ and $E$ is minimal normal in $G$,
	and so it suffices to show that $\norm{E}{E\cap M}=E$.
	First, since $D$ is abelian, $D\cap M\unlhd G$. 
	It follows from $D\cap M<D$ that $D\cap M=1$.
	Also, $|\gfitt{G}:\gfitt{G}\cap M|=|\gfitt{G}M:M|=p^d$.
	We conclude that 
	$$\gfitt{G}=(M\cap \gfitt{G})D\leq  \norm{\gfitt{G}}{E\cap M}D=\norm{E}{E\cap M} \times D\leq \gfitt{G}.$$
	Thus $\norm{E}{E\cap M}=E$, as desired.
\end{proof}

  \begin{prop}\label{F*=Op'}
	Let $G\in\ncpm{d}$ be such that $\oh{p'}{G}=1$ and $|G|_p=p^n>p^d$.
	Write $\overline{G}=G/\Phi(G)$ and $|\Phi(G)|_p=p^s$.
	Assume that $G/\gfitt{G}$ is a $p'$-group. Then one of the following is true.
	
	{\rm (1)} $\overline{G}\in\ncp$. In particular, if $\gfitt{\overline{G}}=\fitt{\overline{G}}$, then $\overline{G}=\overline{H}\ltimes \fitt{\overline{G}}$ is supersolvable where $\overline{H}$ is abelian with exponent dividing by $p-1$.
	
	{\rm (2)}
	$\overline{G}=\overline{H}\ltimes \overline{P}$
	where $\overline{P}$ is a direct product of some minimal normal subgroups 
	$E_1,~E_2,~\dots,~E_t$ with the same order, say $p^e$ where $e>1$, and $e\mid (n-s,d-s)$.
	Moreover, 
	all $E_i$ are isomorphic to an irreducible $\mathbb{F}_p[H]$-module $E$ which is not absolutely irreducible.
\end{prop}
\begin{proof}
	Note that $\oh{p'}{G}=1$, and so $|\Phi(G)|=p^s$ and also $\oh{p'}{\overline{G}}=1$.
	Since every subgroup of $\Phi(G)$ has no complement in $G$, it follows from $G\in\ncpm{d}$ that $p^s<p^d$.
	Applying part (3) of Lemma \ref{element31} to $\overline{G}$, we have $\overline{G}\in\ncpm{d-s}$.
	Also as $\overline{\gfitt{G}}\leq \gfitt{\overline{G}}$, $\overline{G}/\gfitt{\overline{G}}$ is a $p'$-group.
	Thus, without loss of generality, we may assume that $\Phi(G)=1$.
	Hence $\gfitt{G}=\soc{G}$, and $|\gfitt{G}|_p=|G|_p>p^d$.
	Write 
	$\gfitt{G}=E_1\times \cdots\times E_t$ where $E_i$ are minimal normal subgroups of $G$.
	Since $\oh{p'}{G}=1$, $p\mid |E_i|$ for $1\leq i\leq t$. 
	Also by part (2) of Lemma \ref{element31}, $|E_i|_p\leq p^d$ for $1\leq i\leq t$.


	
\smallskip
   
   We claim first that the orders of all minimal normal subgroups of $G$ have equal $p$-part, say $p^e$.
   Let $G$ be a counterexample with minimal possible sum $|G|+d$.
   Since there exists at least one minimal normal subgroup of $G$ with $p$-parts less than $p^d$,
	by part (3) of Lemma \ref{element31} $F^*(G)=D\times E$ where $D$ and $E$ are both minimal normal in $G$. 
	Without loss of generality, we may assume that $|E|_p=p^e<p^d$.
	Since $|E|_p=p^e<p^d$, by part (3) of Lemma \ref{element31} $G/E\in\ncpm{d-e}$.
	It follows by part (2) of Lemma \ref{element31} that $|D|_p=|DE/E|_p\leq p^{d-e}$ as $DE/E$ being minimal normal in $G/E$.
    We conclude that $p^d<|F^*(G)|_p=|D|_p|E|_p\leq p^d$, a contradiction.
	Thus the orders of all minimal normal subgroups of $G$ share equal $p$-part.
	As a consequence, $e\mid n$. 
	
	Write $d=ke+r$ where $0\leq r<e$, and let $N=E_1\times\cdots\times E_k$ where $k\leq t=\frac{n}{e}$.
	If $r>0$, then $G/N\in\ncpm{r}$ by part (3) of Lemma \ref{element31}.
	Now $E_{t}N/N$ is minimal normal in $G/N$, it follows from part (2) Lemma \ref{element31} that $p^e=|E_{t}|_p=|E_{t}N/N|_p\leq p^r$ which contradicts to $r<e$. 
	Consequently, $e\mid (d,n)$.


	
	\smallskip

    We claim next that if $e=1$, then (1) holds.
    Since $|G|_p>p^d\geq p$, by \cite[Proposition E]{qian2015}, it suffices to prove $G\in\ncp$.
	Since $e=1$, we have $t=n>d$. 
	Let $N=E_1\times E_2\times \cdots\times E_d$ and $N_i$ the direct product of groups in 
	$\{E_j\mid 1\leq j\leq d~\text{and}~j\neq i\}$.
    As $|N_i|_p=p^{d-1}$, by part (3) of Lemma \ref{element31} $G/N_i\in\ncp$.
	Also, observing that $\bigcap_{i=1}^d N_i=1$, we have that 
  \[
	  G=G/\bigcap_{i=1}^d N_i\lesssim \bigtimes_{i}^d G/N_i,
  \]
    so we conclude that $G\in\ncp$ since the direct product of two nontrivial $\cp$-groups is still a nontrivial $\cp$-group.



	\smallskip	

     Finally, we show that if $e>1$, then (2) holds.
	 Let $P\in\syl{p}{G}$.
	By Lemma \ref{equalppart}, $\E G=1$, and hence $P=\gfitt{G}=\fitt{G}=\soc{G}$.
	 Since $\Phi(G)=1$, it follows from \cite[Lemma 2.4]{qian2020} that $\Phi(P)=\Phi(G)\cap P=1$, so $P$ is elementary abelian.
     Applying \cite[Proposition E]{qian2015} to $G$, we have that $E_i$ is isomorphic to an $\mathbb{F}_p[H]$-module $E$ for $1\leq i\leq t$.
     Finally, one readily check by \cite[Corollary 2.10, Lemma 2.9(5)]{qian2020} that $E$ is not absolutely irreducible.
\end{proof}


\begin{lem}\label{end}
	Let $H$ be a group
	and $V$ a faithful irreducible $\mathbb{F}_p[H]$-module.
	Then $\dim_{\mathbb{F}_p}(\mathrm{End}_{\mathbb{F}_p[H]}(V))\mid \dim_{\mathbb{F}_p}(V)$. Moreover, $\dim_{\mathbb{F}_p}(\mathrm{End}_{\mathbb{F}_p[H]}(V))=\dim_{\mathbb{F}_p}(V)$ if and only if $H$ is cyclic.
\end{lem}
\begin{proof}
	Let $K$ be an algebraic closure of $\mathbb{F}_p$.
	By \cite[V, Theorem 14.12]{huppert}, 
	\begin{center}
		$K\otimes_{\mathbb{F}_p} V\cong_{K[H]} V_1\oplus\cdots \oplus V_t$ 
	\end{center}
	where $V_i$ are non-isomorphic
	faithful absolutely irreducible $K[H]$-modules with the same dimension.
	Note that by \cite[Chapter B, Lemma 5.4]{doerk1992}  
	\begin{center}
		$K\otimes_{\mathbb{F}_p} \mathrm{End}_{\mathbb{F}_p[H]}(V)\cong_K \mathrm{End}_{K[H]}(K\otimes_{\mathbb{F}_p}V)$
	\end{center}
	as a $K$-linear space, we have 	
	\begin{align*}
		\dim_{\mathbb{F}_p}(\mathrm{End}_{\mathbb{F}_p[H]}(V)) &=\dim_K(K\otimes_{\mathbb{F}_p}\mathrm{End}_{\mathbb{F}_p[H]}(V)) \\
		&=\dim_K\mathrm{End}_{K[H]}(K\otimes_{\mathbb{F}_p}V)\\
		&=\dim_K(\mathrm{End}_{K[H]}(V_1\oplus\cdots \oplus V_t))\\
			 &=\sum_{i=1}^t \dim_K(\mathrm{End}_{K[H]}(V_i))=t.
	\end{align*}
	Since $\dim_{\mathbb{F}_p}(V)=\dim_K(K\otimes_{\mathbb{F}_p}V)=\dim_K(V_1\oplus\cdots\oplus V_t)=t\dim_K(V_1)$,
	it follows that $\dim_{\mathbb{F}_p}(\mathrm{End}_{\mathbb{F}_p[H]}(V))\mid \dim_{\mathbb{F}_p}(V)$.
	Note that $\dim_{\mathbb{F}_p}(\mathrm{End}_{\mathbb{F}_p[H]}(V))=\dim_{\mathbb{F}_p}(V)$ if and only if $\dim_K(V_1)=1$,
	and since $V_1$ is a faithful absolutely irreducible $K[H]$-module, we conclude that $\dim_{\mathbb{F}_p}(\mathrm{End}_{\mathbb{F}_p[H]}(V))=\dim_{\mathbb{F}_p}(V)$ if and only if $H$ is cyclic.
\end{proof}

Now we are ready to prove Theorem B.

\begin{thm}\label{hv}
	Assume that a $p'$-group $H$ acts faithfully on an elementary abelian $p$-group $V$.
	Suppose $G=H\ltimes V$ where $|V|=p^n$.
	Then $G\in\ncpm{d}$ for a fixed positive integer $d<n$ if and only if 
	one of the following statements holds.

	{\rm (1)} $G$ is supersolvable.

	{\rm (2)} $H$ is cyclic and
	$V$ is a homogeneous $\mathbb{F}_p[H]$-module with all its irreducible $\mathbb{F}_p[H]$-submodules having dimension $e$$(>1)$ so that $e\mid (d,n)$.
\end{thm}
\begin{proof}
	($\Leftarrow$) Write $V=V_1\times \cdots \times V_t$ where $V_i$ are minimal normal subgroups of $G$,
	and denote $|V_i|=p^e$,
	we claim that $G\in\ncpm{e}$.

	Suppose $e=1$. 
    Let $N_i$ be the direct product 
	of groups in $\{V_j\mid 1\leq j\leq d~\text{and}~j\neq i\}$.
	Hence $|N_i|_p=p^{d-1}$ and $\bigcap_{i=1}^d N_i=1$.
	Note that $G/N_i\in\ncp$ for $1\leq i\leq d$ by part (3) Lemma \ref{element31} and 
  \[
	  G=G/\bigcap_{i=1}^d N_i\lesssim \bigtimes_{i}^d G/N_i,
  \]
    we conclude by part (1) of Lemma \ref{element31} that $G\in\ncp$ since the direct product of two nontrivial $\cp$-groups is still a nontrivial $\cp$-group.

	Suppose $e>1$. 
	Then $H$ is cyclic and $V_i$ are isomorphic irreducible $\mathbb{F}_p[H]$-submodules for $1\leq i\leq t$.
	We proceed by induction on $|G|$ to prove that $G\in\ncpm{e}$. 
	Let $\Omega$ be the set of all minimal normal subgroups of $G$.
	Note that by Lemma \ref{end} $|\mathrm{End}_{\mathbb{F}_p[H]}(V_i)|=|V_i|=p^e$ as $H$ being cyclic,
	and hence \cite[Chapter B, Proposition 8.2]{doerk1992} implies that 
	\[
		|\Omega|=\frac{p^{et}-1}{p^e-1}=p^{e(t-1)}+\cdots+p^e+1.
	\]
	Let $X\leq G$ be of order $p^e$. Then $X\leq V$ as $V\in\mathrm{Syl}_p(G)$. 
	Since $|\Omega|> |X|$ and every minimal normal subgroup of $G$ shares only identity, there exists $D\in\Omega$ such that $XD=X\times D$.
	Applying Maschke's theorem to $V$, we have $V=D\times E$ where $E$ is normal in $G$.
    Write $L=H\ltimes E$, and by induction we have $G/D\cong L\in\ncpm{e}$.
	Note that $XD/D\leq G/D$ has order $p^e$, and so $G/D=HD/D\ltimes (XD/D\times B/D)$ where $B/D\unlhd G/D$.
	Therefore $B\unlhd G$ and $X\cap B=1$.
	As a consequence, $HB$ is a complement for $X$ in $G$. 
	Thus $G\in\ncpm{e}$.
	By part (5) of Lemma \ref{element31}, $G\in\ncpm{d}$ as $e\mid (d,n)$.

	($\Rightarrow$) Suppose that $G\in\ncpm{d}$. 
	As $\oh{p'}{G}=1$, every nilpotent normal subgroup of $G$ is a $p$-group.
	Also, since $V$ is a normal elementary abelian Sylow $p$-subgroup of $G$, by \cite[Lemma 2.4]{qian2020} we have that $\Phi(G)=\Phi(G)\cap V=\Phi(V)=1$.
    Applying Proposition \ref{F*=Op'} to $G$, we know that either $G$ is supersolvable where $H$ is abelian with exponent dividing $p-1$ or $V$ is a faithful homogeneous $\mathbb{F}_p[H]$-module with all its irreducible submodules being not absolutely irreducible. 
	
	Now we claim that $H$ is abelian.
	Let $G$ be a minimal counterexample. 
	Then $V$ is a faithful homogeneous $\mathbb{F}_p[H]$-module with all its irreducible $\mathbb{F}_p[H]$-submodules not
	absolutely irreducible,
	also since by part (1) of Lemma \ref{element31} $A\ltimes V\in\ncpm{d}$ for $A<H$, 
	every proper subgroup of $H$ is abelian.
	
	Let $W$ be an irreducible $\mathbb{F}_p[H]$-submodule of $V$. 
	Since $V$ is a faithful $\mathbb{F}_p[H]$-module which is
	the direct product of $t$ copies of $W$ with $t>1$, 
	we have that $W$ is a faithful irreducible $\mathbb{F}_p[H]$-module which is also not absolutely irreducible.
    Write $e=\dim_{\mathbb{F}_p}W$, we have that $G\in\ncpm{e}$ by \cite[Corollary 2.10, Lemma 2.9(4)]{qian2020}.
	By the minimality of $G$, we may assume $t=2$ and write $V=W\times U$ where $U$ is also a faithful irreducible $\mathbb{F}_p[H]$-submodule of $V$ which is isomorphic to $W$ as an $\mathbb{F}_p[H]$-module.

	Since every proper subgroup of $H$ is abelian, $H$ is solvable.
	We claim that $H$ possesses an abelian normal subgroup $A$ with a prime index.  
	To see that, it suffices to show such $A$ exists when $H$ is not nilpotent.
	Now $\fitt{H}<H$ is abelian. 
	Write $A=\fitt{H}$.
	Since $\cent{H}{A}=A$, we conclude that $|H:A|=r$ where $r$ is a prime, as claimed.
	By Clifford's theorem, 
	$W_A$ is either an irreducible $\mathbb{F}_p[A]$-module or
	the direct product of $r$ non-isomorphic irreducible $\mathbb{F}_p[A]$-submodules having same dimension (see \cite[Chapter B, Theorem 7.3]{doerk1992}).
	
	Suppose that $W_A$ is not an irreducible $\mathbb{F}_p[A]$-module. 
	Since $A\ltimes V\in\ncpm{d}$ where $V_A$ is not homogeneous as an $\mathbb{F}_p[A]$-module, by the minimality of $G$, $A$ is abelian with exponent dividing $p-1$. 
	As a consequence, every irreducible $\mathbb{F}_p[A]$-submodule of $W_A$ has dimension 1.
	Now $\dim_{\mathbb{F}_p}W=r$, 
	we know from Lemma \ref{end} that 
	\begin{center}
		$\dim_{\mathbb{F}_p}(\mathrm{End}_{\mathbb{F}_p[H]}(W))=\dim_{\mathbb{F}_p}W=r$ 
	\end{center}
	since $W$ is not absolutely irreducible as $\mathbb{F}_p[H]$-module.
	Consequently, by Lemma \ref{end} $H$ is cyclic, a contradiction.
	
	Suppose that $W_A$ is an irreducible $\mathbb{F}_p[A]$-module.
	Since $V_A$ is a faithful homogeneous $\mathbb{F}_p[A]$-module, $W_A$ is also a faithful irreducible $\mathbb{F}_p[A]$-module. 
	It follows that $A$ is cyclic, so by Schur's lemma and Wedderburn's little theorem $\mathrm{End}_{\mathbb{F}_p[A]}(W_A)$ is a finite field.
    Recall that $e=\dim_{\mathbb{F}_p}(W)$,
    we have by Lemma \ref{end} that 
    \begin{center}
		$\dim_{\mathbb{F}_p}\mathrm{End}_{\mathbb{F}_p[A]}(W_A)=\dim_{\mathbb{F}_p}(W_A)=e$, 
	\end{center}
	and hence $\mathrm{End}_{\mathbb{F}_p[A]}(W_A)\cong \mathbb{F}_{p^e}$.
    Since $W$ and $U$ are isomorphic faithful $\mathbb{F}_p[H]$-modules,
	we may identify $H$ as a subgroup of $\GL(e,p)$,
	and also we may assume that $W$ and $U$ are isomorphic faithful $\mathbb{F}_p[\GL(e,p)]$-modules.
	Then $G\leq \GL(e,p)\ltimes V$.
	Write $T=\cent{\GL(e,p)}{A}\ltimes V$.
	Note that $\mathrm{End}_{\mathbb{F}_p[A]}(W_A)=\cent{\GL(e,p)}{A}\cup \{ 0\}$, and so $\cent{\GL(e,p)}{A}\cong \mathbb{F}_{p^e}^{\times}$ is cyclic.
	Since $A\leq \cent{\GL(e,p)}{A}\leq \GL(e,p)$, $W$ and $U$ are isomorphic faithful irreducible $\mathbb{F}_p[\cent{\GL(e,p)}{A}]$-modules,
	and hence $V$ is a faithful homogeneous $\mathbb{F}_p[\cent{\GL(e,p)}{A}]$-module with its irreducible components having dimension $e$.
    Then by the other implication of this lemma $T\in\ncpm{e}$,
	and hence by part (1) of Lemma \ref{element31} $\cent{\GL(e,p)}{H}\ltimes V$, as a subgroup of $T$, is also a nontrivial $\cpm{e}$-group.  
    Write $\Gamma_0=\cent{\GL(e,p)}{H}$ and $a=\dim_{\mathbb{F}_p}\mathrm{End}_{\mathbb{F}_p[H]}(W)$.
	Since $H$ is nonabelian and $W$ is not absolutely irreducible,
	it follows from Lemma \ref{end} that $1<a<e$.
    Note that $\Gamma_0\leq \cent{\GL(e,p)}{A}$,
	and so $W$ is an $\mathbb{F}_p[\Gamma_0]$-module.
	Let $Z$ be an irreducible $\mathbb{F}_p[\Gamma_0]$-submodule of $W$.
    Since $\Gamma_0=\cent{\GL(e,p)}{H}=\mathrm{End}_{\mathbb{F}_p[H]}(W)^{\times}\cong \mathbb{F}_{p^a}^{\times}$,
    it follows from Lemma \ref{end} that $\dim_{\mathbb{F}_p}Z=\dim_{\mathbb{F}_p}\mathrm{End}_{\mathbb{F}_p[\Gamma_0]}(Z)=\dim_{\mathbb{F}_p}\mathbb{F}_{p^a}=a$.
	Observe that $\Gamma_0=\mathrm{End}_{\mathbb{F}_p[H]}(W)^{\times}$,
	and so $\mathbb{F}_p[\Gamma_0]=\mathrm{End}_{\mathbb{F}_p[H]}(W)$.
	Let $0\neq w_0\in Z$ and write $\mathbb{K}=\mathrm{End}_{\mathbb{F}_p[H]}(W)$, and let $X=\langle w_0^k, w_0^{k}f(w_0)\mid k\in \mathbb{K}\rangle$ where $f:W\rightarrow U$
	is an $\mathbb{F}_p[H]$-isomorphism, and hence $X=\langle w_0^{k}\mid k\in \mathbb{K}\rangle\times \langle f(w_0)\rangle$.
    Since $w_0\in Z$ and $Z$ is an irreducible $\mathbb{F}_p[\Gamma_0]$-module where $\mathbb{K}=\mathbb{F}_p[\Gamma_0]$, 
	we have $X=Z\times \langle f(w_0)\rangle$. 
	Notice that by \cite[Chapter B, Proposition 8.2]{doerk1992} $X$ shares a nontrivial common element with every minimal normal subgroups of $G$.
	However, $|X|=p^{a+1}\leq p^e$ which contradicts $G\in\ncpm{e}$.
	\end{proof}

Finally, we prove Theorem~A, which we state again.

\begin{thm}
	Let $G$ be a finite group such that $|G|_p\geq p^{2d}\geq p^2$.
	Assume $\oh{p'}{G}=1$.
	Then $G\in\ncpm{d}$ if and only if one of the following is true.

	{\rm (1)} $G\in\ncp$. 

	{\rm (2)} $G=H\ltimes P$ where $H$ is a cyclic Hall $p'$-subgroup and $P\in\syl{p}{G}$ is a faithful homogeneous $\mathbb{F}_p[H]$-module with all its 
	irreducible submodules having dimension $e$$(>1)$ so that $e\mid (d,\log_p|P|)$.
\end{thm}
\begin{proof}
	($\Rightarrow$) 
	Let $P\in\mathrm{Syl}_p(G)$ where $|P|=p^n$. Since $2d\leq n$, by \cite[Proposition F]{qian2015} $P$ is elementary abelian, and therefore by Lemma \ref{ea} $\Phi(G)=1$ and $P\leq\gfitt{G}$.
	Application of Proposition \ref{F*=Op'} and Theorem \ref{hv} to $G$ yields either (1) or (2).
	
	($\Leftarrow$) If $G\in\ncp$, 
	then by part (5) of Lemma \ref{element31} $G\in\ncpm{d}$ as $p^d\mid |G|$.
	If (2) holds, then $G\in\ncpm{d}$ by Theorem \ref{hv}. 
\end{proof}

\section{Applications}

Let $G$ be a finite group, and let $A \leq G$ and $K \leq H \leq G$ with $K, H$
normal in $G$. 
Following \cite[VI, Definition 11.5]{huppert}, we say that $A$ \emph{covers} $H/K$ if
$AH = AK$, and $A$ \emph{avoids} $H/K$ if $A \cap H = A \cap K$. 
Following \cite{qian2020}, we say that
a subgroup $A$ of $G$ has the \emph{partial cover-avoidance property} in $G$ and call $A$ a \emph{partial
CAP-subgroup} of $G$, if there exists a chief series of $G$ such that $A$ either covers or avoids each chief factor of the chief series.

\bigskip

\noindent\textbf{HY($p^d$)} \emph{Let $P \in \mathrm{Syl}_p(G)$ and $p^d$ be a given prime power such that $1 < p^d < |P|$. Assume
that all subgroups of $G$ of order $p^d$ are partial CAP-subgroups, and assume further that
all cyclic subgroups of $G$ of order $4$ are partial CAP-subgroups when $p^d = 2$ and $P$ is
nonabelian.}

\bigskip

\begin{cor}
	Assume that $G$ satisfies $\mathrm{HY}(p^d)$ with $\oh{p'}{G} = 1$ and $p$-rank larger than $1$. 
	Then $G = H\ltimes P$ where $H \in \mathrm{Hall}_{p'}(G)$ and $P \in \mathrm{Syl}_p(G)$; furthermore, let $V$ be an irreducible $\mathbb{F}_p[H]$-submodule of $P/\Phi(P)$ and write
	\[
		\dim_{\mathbb{F}_p} V = e, d' = d-\log_p |\Phi(P)|, n' = \log_p |P/\Phi(P)|,
	\]
	then the following statements hold.

	{\rm (1)} $P/\Phi(P)$ is a homogeneous $\mathbb{F}_p[H]$-module, while $V$ is not absolutely irreducible.

	{\rm (2)} $d' \geq e \geq 2$, $e \mid (d', n')$, also $G/\Phi(P)$ satisfies $\mathrm{HY}(p^e )$.

	{\rm (3)} $H$ is cyclic.
\end{cor}
\begin{proof}
	By \cite[Theorem A]{qian2020}, $G = H\ltimes P$ where $H\in\mathrm{Hall}_{p'}(G)$ and $P\in\syl{p}{G}$. 
	Observe that $d'\geq 2$ by \cite[Lemma 2.8]{qian2020}, so $G/\Phi(P)$ satisfies HY$(p^{d'})$.
    Also since $P/\Phi(P)\in\syl{p}{G/\Phi(P)}$ is elementary abelian, 
	\cite[Corollary 2.10]{qian2020} implies that $G/\Phi(P)$ satisfies HY$(p^{d'})$ if and only if $G/\Phi(P)\in\ncpm{d'}$. 
    So the results follow directly by Theorem B.
\end{proof}

The above corollary improve \cite[Theorem A$'$]{qian2020} in which $H$ is proved to be supersolvable whose Sylow subgroups are all abelian and the Fitting
subgroup of $H$ is cyclic.


\end{document}